\documentclass[12pt,a4paper]{article}%
\usepackage[active]{srcltx}
\usepackage[utf8]{inputenc}
\usepackage{amsmath}
\usepackage{amsfonts}
\usepackage{amssymb}
\usepackage{graphicx}%
\providecommand{\U}[1]{\protect\rule{.1in}{.1in}}
\newtheorem{theorem}{Theorem}

\newtheorem{corollary}[theorem]{Corollary}

\newtheorem{example}[theorem]{Example}

\newtheorem{lemma}[theorem]{Lemma}

\newtheorem{proposition}[theorem]{Proposition}
\newtheorem{remark}[theorem]{Remark}

\newenvironment{proof}[1][Proof]{\noindent\textbf{#1.} }{$\hfill\Box$}
\begin{document}

\title{\noindent Solutions of the Cheeger problem via torsion functions {\thanks{2000
Mathematics Subject Classification: }}}
\author{\textbf{{\large H. Bueno and G. Ercole}}\thanks{The authors were supported in
part by FAPEMIG and CNPq, Brazil.}\\\textit{{\small Departamento de Matem\'{a}tica}}\\\textit{{\small Universidade Federal de Minas Gerais}}\\\textit{{\small Belo Horizonte, Minas Gerais, 30.123.970, Brazil}}\\\textit{{\small e-mail: hamilton@mat.ufmg.br}}, \thinspace
\textit{{\small grey@mat.ufmg.br}}}
\maketitle

\begin{abstract}
The Cheeger problem for a bounded domain $\Omega\subset\mathbb{R}^{N}$, $N>1$
consists in minimizing the quotients $|\partial E|/|E|$ among all smooth
subdomains $E\subset\Omega$ and the Cheeger constant $h(\Omega)$ is the
minimum of these quotients. Let $\phi_{p}\in C^{1,\alpha}\left( \overline
{\Omega}\right) $ be the $p$-torsion function, that is, the solution of
torsional creep problem $-\Delta_{p}\phi_{p}=1$ in $\Omega$, $\phi_{p}=0$ on
$\partial\Omega$, where $\Delta_{p}u:=\operatorname{div}(|\nabla
u|^{p-2}\nabla u)$ is the $p$-Laplacian operator, $p>1$. The paper emphasizes
the connection between these problems. We prove that $\lim_{p\rightarrow1^{+}%
}(\|\phi_{p}\|_{L^{\infty}(\Omega)})^{1-p}=h(\Omega)=\lim_{p\rightarrow1^{+}%
}(\|\phi_{p}\|_{L^{1}(\Omega)})^{1-p}$. Moreover, we deduce the relation
$\lim_{p\to1^{+}}\|\phi_{p}\|_{L^{1}(\Omega)}\geq C_{N}\lim_{p\to1^{+}}%
\|\phi_{p}\|_{L^{\infty}(\Omega)}$ where $C_{N}$ is a constant depending only
of $N$ and $h(\Omega)$, explicitely given in the paper. An eigenfunction $u\in
BV(\Omega)\cap L^{\infty}(\Omega)$ of the Dirichlet $1$-Laplacian is obtained
as the strong $L^{1}$ limit, as $p\rightarrow1^{+}$, of a subsequence of the
family $\{\phi_{p}/\|\phi_{p}\|_{L^{1}(\Omega)}\}_{p>1}$. Almost all $t$-level
sets $E_{t}$ of $u$ are Cheeger sets and our estimates of $u$ on the Cheeger
set $|E_{0}|$ yield $|B_{1}|\,h(B_{1})^{N}\leq|E_{0}|h(\Omega)^{N}$, where
$B_{1}$ is the unit ball in $\mathbb{R}^{N}$. For $\Omega$ convex we obtain
$u=|E_{0}|^{-1}\chi_{E_{0}}$.

\end{abstract}

\section{Introduction}

In this paper we consider the minimization problem
\begin{equation}
h(\Omega)=\min_{E\subset\Omega}\frac{|\partial E|}{|E|},\label{cheginf}%
\end{equation}
known as the \textit{Cheeger problem}. Here $\Omega\subset\mathbb{R}^{N}$
($N>1$) is smooth and bounded domain, the quotients $|\partial E|/|E|$ are
evaluated among all smooth subdomains $E\subset\Omega$ and the quantities
$|\partial E|$ and $|E|$ denote, respectively, the $(N-1)$-dimensional
Lebesgue perimeter of $\partial E$ and the $N$-dimensional Lebesgue volume of
$E$.

The value $h(\Omega)$ is known as the \textit{Cheeger constant} of $\Omega$
and a corresponding minimizing subdomain $E$ is called a \textit{Cheeger set}
of $\Omega$.

Cheeger sets have importance in the modeling of landslides, see
\cite{Hild,Iunescu}, or in fracture mechanics, see \cite{Keller}.

On its turn, the Cheeger constant of $\Omega$ itself offers a lower bound (see
\cite{Grieser, Lefton}) for the first eigenvalue $\lambda_{p}(\Omega)$ of the
$p$-Laplacian operator $\Delta_{p} u:=\operatorname{div}\left( |\nabla
u|^{p-2}\nabla u\right) $, $p>1$, with homogeneous Dirichlet data, that is,
$\lambda_{p}(\Omega)$ is the only positive real number that satisfies
\begin{equation}
\left\{
\begin{array}
[c]{rcll}%
-\Delta_{p}u_{p} & = & \lambda_{p}u_{p}^{p-1}, & \text{in }\Omega\\
u_{p} & = & 0, & \text{on }\partial\Omega
\end{array}
\right. \label{eigenproblem}%
\end{equation}
for some positive function $u_{p}\in W_{0}^{1}(\Omega)\setminus\{0\}$.

It is well-known that
\begin{equation}
\lambda_{p}(\Omega)=\frac{\int_{\Omega}|\nabla u_{p}|^{p}\,dx}{\int_{\Omega
}u_{p}^{p}\,dx}=\inf\left\{ \frac{\int_{\Omega}|\nabla u|^{p}\,dx}%
{\int_{\Omega}|u|^{p}\,dx}\,:\,u\in W_{0}^{1}(\Omega)\setminus\{0\}\right\}
.\label{lambdamin}%
\end{equation}

A strong connection between the solutions of the eigenvalue problem
(\ref{eigenproblem}) and of the Cheeger problem (\ref{cheginf}) became evident
from the remarkable work \cite{Kaw2} by Kawohl and Fridman. In that paper they
proved that
\begin{equation}
h(\Omega)=\lim_{p\rightarrow1^{+}}\lambda_{p}(\Omega)\label{kaw}%
\end{equation}
and that $L^{\infty}$-normalized family $\left\{ u_{p}\right\} $ of positive
eigenfunctions converges in $L^{1}$ (up to subsequences), as $p\rightarrow
1^{+}$, to a bounded function $u$ whose level sets $E_{t}=\left\{ x\in
\Omega\,:\,u(x)>t\right\} $ are Cheeger sets for almost all $0\leq t\leq1$.
Moreover, if $\Omega$ is convex they argued that $E_{t}=E_{0}$ for almost all
$0<t\leq1$ and $u=c\chi_{E_{0}}$ ($\chi_{A}$ denotes the characteristic
function of $A$). We remark that Cheeger sets are unique if $\Omega$ is convex
(see \cite{AlterCaselles, Caselles, Stred}).

The function $u$ built in \cite{Kaw2} solves the eigenvalue problem for the
$1$-Laplacian $\Delta_{1}=\operatorname{div}(\nabla u/|\nabla u|)$:
\begin{equation}
\left\{
\begin{array}
[c]{rcll}%
-\Delta_{1} & = & h(\Omega), & \text{in }\Omega\\
u & = & 0, & \text{on }\partial\Omega
\end{array}
\right. \label{1lap}%
\end{equation}
formally deduced by taking $p=1$ in (\ref{eigenproblem}) and keeping
(\ref{kaw}) in mind. Apparently inspired by the variational characterization
of $\lambda_{p}(\Omega)$ in (\ref{lambdamin}), Kawohl and Fridman \cite{Kaw2}
have reformulated (\ref{1lap}) as a minimizing problem of quotients in the
$BV(\Omega)$ space. Then, after verifying that $\left\{ u_{p}\right\} $ is a
bounded family in $BV(\Omega)$ and applying properties of this space, they
proved the existence of a solution $u\in BV(\Omega)$ as mentioned above.
Moreover, in \cite{Kaw2} the authors clarified the equivalence between the
problems (\ref{eigenproblem}) and (\ref{cheginf}) as well as presented some
examples and properties of the Cheeger sets related to uniqueness, regularity
and convexity.

A $BV$-formulation had already appeared in \cite{Kaw1} for the operator
$\Delta_{1}$, where some free boundary problems were introduced and
interrelated through a minimization problem for a certain energy functional
$J_{1}$ that generalizes, for $p=1$, the \textit{torsional creep} problem
\begin{equation}
\left\{
\begin{array}
[c]{rcll}%
-\Delta_{p}\phi_{p} & = & 1, & \text{in }\Omega\\
\phi_{p} & = & 0, & \text{on }\partial\Omega.
\end{array}
\right. \label{p-tors}%
\end{equation}
However, the existence of Cheeger sets and the obtention of the Cheeger
constant were not treated in that paper.

Since \cite{Kaw1} and \cite{Kaw2} the variational treatment of problems
involving $\Delta_{1}$ in the $BV(\Omega)$ space has been naturally adopted in
the literature \cite{AlterCaselles, Butazzo1, Cicalese, Hild, Iunescu,KawNov}.
We refer to \cite{Carlier} for a complete treatment of a more general Cheeger problem.

Our goal in this paper is to emphasize the strong connection between solutions
of the Cheeger problem and the family $\left\{ \phi_{p}\right\} $ of the
$p$-\textit{torsion functions}, that is, solutions of the torsional creep
problem (\ref{p-tors}).

The major part of our approach connects (\ref{p-tors}) directly to
(\ref{cheginf}) and some relations can be used as alternative estimates for
$\lambda_{p}(\Omega)$ and $h(\Omega)$.

We prove that
\begin{equation}
\lim_{p\rightarrow1^{+}}\frac{1}{\|\phi_{p}\|_{\infty}^{p-1}}=h(\Omega
)=\lim_{p\rightarrow1^{+}}\frac{1}{\|\phi_{p}\|_{1}^{p-1}}\label{characheeg}%
\end{equation}
where $\|\phi_{p}\|_{\infty}$ and $\|\phi_{p}\|_{1}$ denote, respectively, the
$L^{\infty}$ norm and the $L^{p}$ norm of the $p$-torsion function $\phi_{p}$.

We also deduce a Cheeger inequality involving $\Vert\phi_{p}\Vert_{\infty}$
and $\Vert\phi_{p}\Vert_{1}$:
\begin{equation}
|B_{1}|\left(  \frac{h(B_{1})}{h(\Omega)}\right)  ^{N}=\omega_{N}\left(
\frac{N}{h(\Omega)}\right)  ^{N}\leq\liminf_{p\rightarrow1^{+}}\frac{\Vert
\phi_{p}\Vert_{1}}{\Vert\phi_{p}\Vert_{\infty}}\label{bounds}%
\end{equation}
where $\omega_{N}=|B_{1}|$ is the volume of the unit ball $B_{1}%
\subset\mathbb{R}^{N}$.

By exploring (\ref{characheeg}) and standard properties of $BV$-functions we
obtain, as in \cite{Kaw2} or \cite[Section 2]{Butazzo1}, the $L^{1}$
convergence (up to subsequences), when $p\rightarrow1^{+}$, of the family
$\left\{ \frac{\phi_{p}}{\|\phi_{p}\|_{1}}\right\} _{p\geq1}$ for a solution
$u\in BV(\Omega)\cap L^{\infty}(\Omega)$ of (\ref{1lap}). In view of general
properties of solutions of (\ref{1lap}) (see \cite{Kaw2} or \cite{Carlier})
the $t$-level sets $E_{t}$ of this function are Cheeger set for almost $0\leq
t\leq\|u\|_{\infty}$ and, moreover, if $\Omega$ is convex, $E_{t}=E_{0}$ for
almost $0\leq t\leq\|u\|_{\infty}$ and $u=\frac{\chi_{E_{0}}}{|E_{0}|}$.

As consequence of the estimate (\ref{bounds}) the function limit $u$
satisfies
\[
0\leq u\leq\omega_{N}^{-1}\left( \dfrac{h(\Omega)}{N}\right) ^{N}\text{ \ in
\ }\Omega
\]
implying the following estimate for the Cheeger set $E_{0}$:
\[
|B_{1}|h(B_{1})^{N}\leq|E_{0}|h(\Omega)^{N}.
\]
This estimate is optimal when $\Omega$ is a ball, the known case where
$\Omega$ is its Cheeger set itself.

Alternatively, the same convergence result can be proved for the family
$\left\{ \frac{\phi_{p}}{\|\phi_{p}\|_{\infty}}\right\} _{p\geq1}$.

To obtain the characterizations of $h(\Omega)$ in (\ref{characheeg}) we
explore some properties of the energy functional $J_{p}$ associated to the
torsional creep problem (\ref{p-tors}) and deduce an estimate relating
$h(\Omega)$ and $\|\phi_{p}\|_{1}$, see equation (\ref{hright}). The first
characterization in (\ref{characheeg}) was possible thanks to the estimate
(\ref{bounds}) that we prove inspired by the arguments of \cite[Chap. 2 Sect.
5]{Ladyz} (see also \cite[Theor. 2]{Bandle}). However, in order to handle some
limits as $t\rightarrow1^{+}$ we had to develop some auxiliary estimates with
explicit $p$-dependence.

We also provide a simpler proof of (\ref{characheeg}), if $\Omega$ is convex.
For this we use Schwarz symmetrization and explore the concavity of $\phi
_{p}^{1-\frac{1}{p}}$ (see \cite{Saka}), which, taking into account the
convexity of $\Omega$, can be used to justify the well-known convexity of the
unique Cheeger set.

The paper is organized as follows: in Section \ref{CCC} we prove
(\ref{characheeg}) and the given estimates of the Cheeger constant $h(\Omega
)$. In Section \ref{CD} we consider the special case of a convex domain
$\Omega$, where an alternative proof of (\ref{characheeg}) is obtained and
also some estimates of the Cheeger constant in terms of Beta and Gamma
functions. Part of the final Section \ref{CS} is written for the convenience
of the reader and reproduces the current variational approach in the $BV$
space for the Cheeger problem (\ref{cheginf}) and some of the main results of
this theory, following \cite{Carlier}. Then, we apply this approach to obtain
Cheeger sets as level sets of a solution $u\in BV(\Omega)\cap L^{\infty
}(\Omega)$ of (\ref{1lap}) and state the estimate $|B_{1}|h(B_{1})^{N}%
\leq|E_{0}|h(\Omega)^{N}$ for the Cheeger set $E_{0}$. We end the paper by
illustrating this estimate for a plane square.

\section{Characterizations of the Cheeger constant}

\label{CCC} In this section we prove that
\[
\lim_{p\rightarrow1^{+}}\frac{1}{\|\phi_{p}\|_{\infty}^{p-1}}=h(\Omega
)=\lim_{p\rightarrow1^{+}}\frac{1}{\|\phi_{p}\|_{1}^{p-1}}
\]
where $\phi_{p}$ is the $p$-torsion function of $\Omega$, that is, the
solution of (\ref{p-tors}).

It is easy to verify that the $p$-torsion function of a ball $B_{R}$ of radius
$R$ with center at the origin is the radially symmetric function
\begin{equation}
\Phi_{p}(r)=\frac{p-1}{p}N^{-\frac{1}{p-1}}\left( R^{\frac{p}{p-1}}%
-r^{\frac{p}{p-1}}\right) ,\text{ \ }r=|x|\leq R.\label{Rtorsion}%
\end{equation}

Positivity, boundedness and $C^{1,\beta}$-regularity follow from this
expression. Hence, as consequence of the comparison principle and regularity
theorems (see \cite{DiBenedeto, Lieberman, Tolks}) these properties are easily
transferred to the $p$-torsion function of a general bounded domain $\Omega$.
Thus, one has $\phi_{p}>0$ in $\Omega$,
\[
\|\phi_{p}\|_{\infty}\leq\frac{p-1}{p}N^{-\frac{1}{p-1}}R^{\frac{p}{p-1}}
\]
for any $R>0$ such that $\Omega\subset B_{R}$ and $\phi_{p}\in C^{1,\beta
}\left( \overline{\Omega}\right) \cap W_{0}^{1,p}(\Omega)$ for some
$0<\beta<1$.

It follows from (\ref{p-tors}) that
\begin{equation}
\int_{\Omega}|\nabla\phi_{p}|^{p-2}\nabla\phi_{p}\cdot\nabla v\,dx=\int
_{\Omega}v\,dx\text{ \ for all }v\in W_{0}^{1,p}(\Omega)\label{aux1}%
\end{equation}
which yields, by taking $v=\phi_{p}$,
\begin{equation}
\int_{\Omega}|\nabla\phi_{p}|^{p}\,dx=\int_{\Omega}\phi_{p}%
\,dx.\label{torsgrad}%
\end{equation}

Moreover, a standard variational argument shows that ${\phi}_{p}$ minimizes
the strictly convex energy functional $J_{p}\colon W_{0}^{1,p}(\Omega
)\rightarrow\mathbb{R}$ given by
\begin{equation}
J_{p}(u)=\frac{1}{p}\int_{\Omega}|\nabla u|^{p}\,dx-\int_{\Omega
}u\,dx.\label{Jp}%
\end{equation}

\begin{lemma}
\label{Jplemma} Let $\Omega\subset\mathbb{R}^{N}$ be a bounded, smooth domain.
If $\varphi\in W_{0}^{1,p}(\Omega)$ is nonnegative in $\Omega$ and such that
$\int_{\Omega}|\nabla\varphi|\,dx>0$, then
\begin{equation}
\liminf_{p\rightarrow1^{+}}\|\phi_{p}\|_{1}^{p-1}\geq\frac{\int_{\Omega
}\varphi\,dx}{\int_{\Omega}|\nabla\varphi|\,dx},\label{lemint}%
\end{equation}
where $\phi_{p}$ is the $p$-torsion function of $\Omega$ and $\|\cdot\|_{1}$
stands for the $L^{1}$-norm.
\end{lemma}

\begin{proof}
Since $\phi_{p}$ is a minimizer of the functional energy $J_{p}$ in
$W_{0}^{1,p}(\Omega)$ it follows from (\ref{aux1}) and (\ref{Jp}) that for all
$u\in W_{0}^{1,p}(\Omega)$ one has
\[
\left( \frac{1}{p}-1\right) \int_{\Omega}\phi_{p}\,dx=J_{p}(\phi_{p})\leq
\frac{1}{p}\int_{\Omega}|\nabla u|^{p}\,dx-\int_{\Omega}u\,dx.
\]
Thus,
\begin{equation}
\|\phi_{p}\|_{1}\geq\frac{1}{p-1}\left( p\int_{\Omega}u\,dx-\int_{\Omega
}|\nabla u|^{p}\,dx\right) \,dx\text{ \ for all }u\in W_{0}^{1,p}%
(\Omega).\label{nconv2}%
\end{equation}

Now let $\varphi\in W_{0}^{1,p}(\Omega)$ be nonnegative in $\Omega$ and such
that $\int_{\Omega}|\nabla\varphi|\,dx>0$. For a fixed $\epsilon$,
$0<\epsilon<1$, let $c_{p}$ be the positive constant such that
\[
p\int_{\Omega}\varphi\,dx-c_{p}^{p-1}\int_{\Omega}|\nabla\varphi
|^{p}\,dx=\epsilon\int_{\Omega}\varphi\,dx,
\]
that is
\[
c_{p}^{p-1}=(p-\epsilon)\frac{\int_{\Omega}\varphi\,dx}{\int_{\Omega}%
|\nabla\varphi|^{p}\,dx}.
\]

It follows from (\ref{nconv2}) with $u=c_{p}\varphi$ that
\[
\|\phi_{p}\|_{1}\geq\frac{c_{p}}{p-1}\left( p\int_{\Omega}\varphi
\,dx-c_{p}^{p-1}\int_{\Omega}|\nabla\varphi|^{p}\,dx\right) =\frac{\epsilon
c_{p}}{p-1}\int_{\Omega}\varphi\,dx.
\]

Therefore,
\begin{align*}
\liminf_{p\rightarrow1^{+}}\|\phi_{p}\|_{1}^{p-1} & \geq\lim_{p\rightarrow
1^{+}}c_{p}^{p-1}\left( \frac{1}{p-1}\right) ^{p-1}\left( \epsilon\int
_{\Omega}\varphi\,dx\right) ^{p-1}\\
& =\lim_{p\rightarrow1^{+}}c_{p}^{p-1}=(1-\epsilon)\frac{\int_{\Omega}%
\varphi\,dx}{\int_{\Omega}|\nabla\varphi|\,dx}.
\end{align*}
Making $\epsilon\rightarrow0$, (\ref{lemint}) follows.
\end{proof}

\begin{theorem}
\label{Rcheeger} Let $\Omega\subset\mathbb{R}^{N}$ be a bounded, smooth domain
and $\phi_{p}$ its $p$-torsion function. Then
\begin{equation}
h(\Omega)\leq\left( \frac{|\Omega|}{\|\phi_{p}\|_{1}}\right) ^{\frac{p-1}{p}%
}\label{hright}%
\end{equation}
and
\begin{equation}
\lim_{p\rightarrow1^{+}}\frac{1}{\|\phi_{p}\|_{1}^{p-1}}=h(\Omega
).\label{limhright}%
\end{equation}

\end{theorem}

\begin{proof}
The estimate (\ref{hright}) follows from Cavalieri's principle and coarea
formula applied to the $p$-torsion function $\phi_{p}$. In fact, since
$\phi_{p}\in C^{1,\beta}\left( \overline{\Omega}\right) $ we have
\[
\int_{\Omega}\phi_{p}\,dx=\int_{0}^{\|\phi_{p}\|_{\infty}}|A_{t}|\,dt
\]
and
\[
\int_{\Omega}|\nabla\phi_{p}|\,dx=\int_{0}^{\|\phi_{p}\|_{\infty}}|\partial
A_{t}|\,dt
\]
where
\[
A_{t}=\left\{ x\in\Omega\,:\,\phi_{p}(x)>t\right\} .
\]

Therefore, since $h(\Omega)\leq\dfrac{|\partial A_{t}|}{|A_{t}|}$, we have
\[
\int_{\Omega}|\nabla\phi_{p}|\,dx=\int_{0}^{\|\phi_{p}\|_{\infty}}|\partial
A_{t}|\,dt\geq\int_{0}^{\|\phi_{p}\|_{\infty}}h(\Omega)|A_{t}|\,dt=h(\Omega
)\int_{\Omega}\phi_{p}\,dx.
\]
Thus, H\"{o}lder inequality and (\ref{torsgrad}) yield that
\begin{align*}
h(\Omega)  & \leq\frac{\int_{\Omega}|\nabla\phi_{p}|\,dx}{\int_{\Omega}%
\phi_{p}\,dx}\\
& \leq\frac{\left( \int_{\Omega}|\nabla\phi_{p}|^{p}\,dx\right) ^{\frac{1}{p}%
}|\Omega|^{1-\frac{1}{p}}}{\int_{\Omega}\phi_{p}\,dx}=\frac{\left(
\int_{\Omega}\phi_{p}\,dx\right) ^{\frac{1}{p}}|\Omega|^{1-\frac{1}{p}}}%
{\int_{\Omega}\phi_{p}\,dx}=\left( \frac{|\Omega|}{\|\phi_{p}\|_{1}}\right)
^{\frac{p-1}{p}},
\end{align*}
which is (\ref{hright}). It follows then
\[
h(\Omega)\leq\liminf_{p\rightarrow1^{+}}\left( \frac{|\Omega|}{\|\phi
_{p}\|_{1}}\right) ^{\frac{p-1}{p}}=\liminf_{p\rightarrow1^{+}}\frac{1}%
{\|\phi_{p}\|_{1}^{p-1}}.
\]

To complete the proof, we will firstly prove that $\limsup
\limits_{p\rightarrow1^{+}}\frac{1}{\|\phi_{p}\|_{1}^{p-1}}$ is a lower bound
to the quotients $|\partial E|/|E|$ formed by smooth subdomains $E\subset
\subset\Omega$ whose boundary $\partial E$ does not intercept $\partial\Omega$.

Let $E$ such a domain. We approximate the characteristic function of $E$ by a
suitable nonnegative function $\varphi_{\varepsilon}\in W_{0}^{1,p}(\Omega)$
such that $\varphi_{\varepsilon}\equiv1$ on $E$, $\varphi_{\varepsilon}%
\equiv0$ outside an $\varepsilon$-neighborhood of $E$ with $|\nabla
\varphi_{\varepsilon}|=1/\varepsilon$ on an $\varepsilon$-layer outside $E$
($\varphi_{\varepsilon}$ can be taken Lipschitz). Then, for each
$t\in[0,\varepsilon]$, denoting by $\Gamma_{t}$ the $t$-layer outside $E$ (in
a such way that $\Gamma_{0}\cup E=E$), it follows from (\ref{lemint}) that
\begin{align*}
\limsup_{p\rightarrow1^{+}}\frac{1}{\|\phi_{p}\|_{1}^{p-1}} & \leq\frac
{\int_{\Omega}|\nabla\varphi_{\varepsilon}|\,dx}{\int_{\Omega}\varphi
_{\varepsilon}\,dx}\\
& =\frac{\int_{0}^{\varepsilon}\int_{\partial(\Gamma_{t}\cup E)}\frac
{1}{\varepsilon}\,dS_{x}\,dt}{|E|+\int_{\Gamma_{\varepsilon}}\varphi
_{\varepsilon}\,dx}\\
& \leq\frac{\frac{1}{\varepsilon}\left( \int_{0}^{\varepsilon}\,dt\right)
\left( \int_{\partial(\Gamma_{\varepsilon}\cup E)}\,dS_{x}\right) }{|E|}%
=\frac{\int_{\partial(\Gamma_{\varepsilon}\cup E)}\,dS_{x}}{|E|}.
\end{align*}
Therefore, making $\varepsilon\rightarrow0^{+}$, we find
\[
\limsup_{p\rightarrow1^{+}}\frac{1}{\|\phi_{p}\|_{1}^{p-1}}\leq\frac{|\partial
E|}{|E|}.
\]

Now, if $E$ touches $\partial\Omega$, we approximate $E$ by a sequence
$\left\{ t_{n}E\right\} $ of subdomains $t_{n}E_{n}\subset\subset\Omega$ such
that $t_{n}\rightarrow1^{-}$. Since $|t_{n}E_{n}|=t_{n}^{N}|E|$ and
$|\partial(t_{n}E)|=t_{n}^{N-1}|\partial E|$ we have that
\[
\limsup_{p\rightarrow1^{+}}\frac{1}{\|\phi_{p}\|_{1}^{p-1}}\leq\frac
{|\partial(t_{n}E)|}{|t_{n}E|}=\frac{1}{t_{n}}\frac{|\partial E|}{|E|}.
\]
Thus, as $t_{n}\rightarrow1^{-}$ we obtain
\[
\limsup_{p\rightarrow1^{+}}\frac{1}{\|\phi_{p}\|_{1}^{p-1}}\leq\frac{|\partial
E|}{|E|}.
\]

\vspace{-.6cm}
\end{proof}

\begin{remark}
\textrm{In the proof of (\ref{limhright}) another estimate like (\ref{hright})
could be obtained by applying the variational characterization
(\ref{lambdamin}) of $\lambda_{p}(\Omega)$ and the well-known lower bound for
$\lambda_{p}(\Omega)$ in terms of the Cheeger constant $h(\Omega)$ (for $2\neq
p>1$, see \cite{Lefton}):
\[
\left( \frac{h(\Omega)}{p}\right) ^{p}\leq\lambda_{p}(\Omega).
\]
In fact, it follows from (\ref{lambdamin}) that $\lambda_{p}(\Omega
)\leq(|\Omega|/\|\phi_{p}\|_{1})^{p-1}$ (see equation (\ref{npbound}), in the
sequel). Thus,
\[
\left( \frac{h(\Omega)}{p}\right) ^{p}\leq\left( \frac{|\Omega|}{\|\phi
_{p}\|_{1}}\right) ^{p-1}.
\]
The chosen estimate (\ref{hright}) emphasizes the direct connection between
the $p$-torsion functions and the Cheeger constant $h(\Omega)$. Moreover, it
follows from the last inequality that
\[
h(\Omega)\leq p\left( \frac{|\Omega|}{\|\phi_{p}\|_{1}}\right) ^{\frac{p-1}%
{p}}
\]
an estimate that is slightly worse than (\ref{hright}), because
\[
h(\Omega)\leq\left( \frac{|\Omega|}{\|\phi_{p}\|_{1}}\right) ^{\frac{p-1}{p}%
}<p\left( \frac{|\Omega|}{\|\phi_{p}\|_{1}}\right) ^{\frac{p-1}{p}}
\]
for $p>1$. }
\end{remark}

\begin{remark}
\label{touch}\textrm{The approximation argument used at the end of the last
proof shows that any Cheeger set touches $\partial\Omega$. In fact, if a
Cheeger set $E$ does not touch $\partial\Omega$ then we can take
$t_{\varepsilon}=1+\varepsilon>1$ such that $t_{\varepsilon}E\subset\Omega$
with $t_{\varepsilon}E$ touching the boundary $\partial\Omega$. But this leads
to a contradiction since
\[
h(\Omega)\leq\frac{|\partial(t_{\varepsilon}E)|}{|t_{\varepsilon}E|}=\frac
{1}{t_{\varepsilon}}\frac{|\partial E|}{|E|}=\frac{1}{t_{\varepsilon}}%
h(\Omega)<h(\Omega).
\]
}
\end{remark}

\vspace{.3cm}

We recall that if $u$ is a continuous and nonnegative function defined in
$\Omega\subset\mathbb{R}^{N}$ then the Schwarz symmetrization $u^{*}$ of $u$
is the function defined in $\Omega^{*}$ that satisfies (see \cite{Kaw0})
\[
\left\{ x\in\Omega\,:\,u(x)>t\right\} ^{*}=\left\{ x\in\Omega^{*}%
\,:\,u^{*}(x)>t\right\}
\]
for all $t\geq0$, where $A^{*}$ denotes the ball with center at the origin and
same Lebesgue measure as $A$.

Let $\Omega$ be a bounded, smooth domain in $\mathbb{R}^{N}$, $N>1$. The
following lemma is a consequence of Talenti's comparison principle
\cite{Talenti} for the $p$-Laplacian, which says that if $u$ and $U$ are,
respectively, solutions of the Dirichlet problems
\[
\left\{
\begin{array}
[c]{rcll}%
-\Delta_{p}u & = & f\text{ \ in \ }\Omega & \\
u & = & 0\text{ \ on }\partial\Omega &
\end{array}
\right. \quad\text{and}\quad\left\{
\begin{array}
[c]{rcll}%
-\Delta_{p}U & = & f^{*}\text{ \ in \ }\Omega^{*} & \\
U & = & 0\text{\ \ \ on\ }\partial\Omega^{*}, &
\end{array}
\right.
\]
where $f^{*}$ is the Schwarz symmetrization of $f$, then the Schwarz
symmetrization $u^{*}$ of $u$ is bounded above by $U$, that is,
\[
u^{*}\leq U\text{ \ in \ }\Omega^{*}.
\]

\begin{lemma}
\label{talenti} Let $\phi_{p}$ and $\Phi_{p}$ be the $p$-torsion functions of
the domains $\Omega$ and $\Omega^{*}$, respectively. If $\phi_{p}^{*}$ denotes
the Schwarz symmetrization of $\phi_{p}$ then
\[
\phi_{p}^{*}\leq\Phi_{p}\text{ \ in \ }\Omega^{*}=B_{R}.
\]

\end{lemma}

The next result provides localization for $\lambda_{p}(\Omega)$. Moreover it
gives an explicit lower bound to this eigenvalue which will be fundamental to
deduce a uniform (with respect to $p$) upper bound to the quotient
$\dfrac{\|\phi_{p}\|_{1}}{\|\phi_{p}\|_{\infty}}$ and hence to prove that
$\lim\limits_{p\rightarrow1^{+}}\|\phi_{p}\|_{\infty}^{1-p}=h(\Omega)$.

\begin{proposition}
If $\Omega\subset\mathbb{R}^{N}$ is a bounded, smooth domain, then
\begin{equation}
C_{N,p}|\Omega|^{-\frac{p}{N}}\leq\|\phi_{p}\|_{\infty}^{1-p}\leq\lambda
_{p}(\Omega)\leq|\Omega|^{p-1}\|\phi_{p}\|_{1}^{1-p}\label{npbound}%
\end{equation}
where $\lambda_{p}(\Omega)$ and $\phi_{p}$ denote, respectively, the first
eigenvalue of $(\ref{eigenproblem})$ and the $p$-torsion function of $\Omega
$,
\begin{equation}
C_{N,p}=N\omega_{N}^{\frac{p}{N}}\left( \frac{p}{p-1}\right) ^{p-1}\label{CNp}%
\end{equation}
and $\omega_{N}=|B_{1}|$ is the volume of the unit ball in $\mathbb{R}^{N}$.
\end{proposition}

\begin{proof}
The last inequality in (\ref{npbound}) follows from (\ref{lambdamin}) applied
to the function $\phi_{p}$. In fact, by the H\"{o}lder inequality
\[
|\Omega|^{1-p}\left( \int_{\Omega}\phi_{p}\,dx\right) ^{p}\leq|\Omega
|^{1-p}\left[ \left( \int_{\Omega}\phi_{p}^{p}\,dx\right) ^{\frac{1}{p}%
}|\Omega|^{1-\frac{1}{p}}\right] ^{p}=\int_{\Omega}\phi_{p}^{p}\,dx.
\]
Thus,
\begin{align*}
\lambda_{p}(\Omega) & \leq\frac{\int_{\Omega}|\nabla\phi_{p}|^{p}\,dx}%
{\int_{\Omega}\phi_{p}^{p}\,dx}= \frac{\int_{\Omega}\phi_{p}\,dx}{\int
_{\Omega}\phi_{p}^{p}\,dx}\leq\frac{\int_{\Omega}\phi_{p}\,dx}{|\Omega
|^{1-p}\left( \int_{\Omega}\phi_{p}\,dx\right) ^{p}} =\left( \frac{|\Omega
|}{\|\phi_{p}\|_{1}}\right) ^{p-1}.
\end{align*}

The second inequality in (\ref{npbound}) is consequence of applying a
comparison principle to the positive eigenfunction $e_{p}$ (with $\Vert
e_{p}\Vert_{\infty}=1$) and $\phi_{p}$, since both vanish on $\partial\Omega$
and
\[
-\Delta_{p}e_{p}=\lambda_{p}(\Omega)e_{p}^{p-1}\leq\lambda_{p}(\Omega
)=-\Delta_{p}\left(  \lambda_{p}(\Omega)^{\frac{1}{p-1}}\phi_{p}\right)  .
\]
Thus,
\[
0\leq e_{p}\leq\lambda_{p}(\Omega)^{\frac{1}{p-1}}\phi_{p}\ \text{ in }\Omega
\]
and, taking the maximum values of these functions, one obtains $1=\Vert
e_{p}\Vert_{\infty}\leq\lambda_{p}(\Omega)^{\frac{1}{p-1}}\Vert\phi_{p}%
\Vert_{\infty}$ and hence
\begin{equation}
\frac{1}{\Vert\phi_{p}\Vert_{\infty}^{p-1}}\leq\lambda_{p}(\Omega
).\label{lefttors}%
\end{equation}

In order to prove the first inequality in (\ref{npbound}), let $\Phi_{p}$ be
the $p$-torsion function of $\Omega^{*}=B_{R}$, where $B_{R}$ is the ball with
center at the origin and radius $R$ such that $|B_{R}|=|\Omega|$.

According to (\ref{Rtorsion}) we have $\|\Phi_{p}\|_{\infty}=\Phi_{p}\left(
0\right) $ and so
\begin{align*}
\|\Phi_{p}\|_{\infty}^{1-p} & =\left( \frac{p}{p-1}\right) ^{p-1}\frac
{N}{R^{p}}\\
& =\left( \frac{p}{p-1}\right) ^{p-1}\frac{N\omega_{N}^{\frac{p}{N}}}%
{(\omega_{N}R^{N})^{\frac{p}{N}}}=C_{N,p}|B_{R}|^{-\frac{p}{N}} =C_{N,p}%
|\Omega|^{-\frac{p}{N}}%
\end{align*}
where $C_{N,p}$ is defined by (\ref{CNp}).

It follows from Lemma \ref{talenti} that
\[
\phi_{p}^{*}\leq\Phi_{p}\text{ \ in \ }\Omega^{*}.
\]
Thus,
\[
C_{N,p}|\Omega|^{-\frac{p}{N}}=\|\Phi_{p}\|_{\infty}^{1-p}\leq\|\phi_{p}%
^{*}\|_{\infty}^{1-p}=\|\phi_{p}\|_{\infty}^{1-p},
\]
since the Schwarz symmetrization preserves the sup-norm.
\end{proof}

\begin{remark}
\textrm{The following inequalities are also given in Kawohl and Fridman
\cite[Corollary 15]{Kaw2}
\[
N\left( \frac{\omega_{N}}{|\Omega|}\right) ^{\frac{1}{N}}\leq h(\Omega)
\ \ \text{and }\ \lim\limits_{p\rightarrow\infty}\left( \lambda_{p}%
(\Omega)\right) ^{\frac{1}{p}}\geq\lim\limits_{p\rightarrow\infty}\|\phi
_{p}\|_{\infty}^{1-p}\geq\left( \frac{\omega_{N}}{|\Omega|}\right) ^{\frac
{1}{N}}.
\]
Both follow from (\ref{npbound}). }
\end{remark}

\begin{corollary}
\label{explicitlower} Let $\Omega\subset\mathbb{R}^{N}$ be a bounded and
smooth domain. Then,
\[
\int_{\Omega}|u|^{p}\,dx\leq\frac{|\Omega|^{\frac{p}{N}}}{C_{N,p}}\int
_{\Omega}|\nabla u|^{p}\,dx
\]
for all $u\in W_{0}^{1,p}(\Omega)\setminus\{0\}$, where $C_{N,p}$ is given by
$(\ref{CNp})$.
\end{corollary}

\begin{proof}
It follows from (\ref{npbound}) and of the variational characterization of
$\lambda_{p}(\Omega)$ since
\[
C_{N,p}|\Omega|^{-\frac{p}{N}}\leq\lambda_{p}(\Omega)\leq\frac{\int_{\Omega
}|\nabla u|^{p}\,dx}{\int_{\Omega}|u|^{p}\,dx}.
\]

\vspace{-.6cm}
\end{proof}

\begin{theorem}
\label{linfty} Let $\phi_{p}$ be the $p$-torsion function of the bounded
domain $\Omega\subset\mathbb{R}^{N}$. Then,
\begin{equation}
\liminf_{p\rightarrow1^{+}}\int_{\Omega}\frac{\phi_{p}}{\|\phi_{p}\|_{\infty}%
}\,dx \geq\omega_{N}\left( \frac{N}{h(\Omega)}\right) ^{N}\label{linfty1}%
\end{equation}
and
\begin{equation}
\lim_{p\rightarrow1^{+}}\left( \int_{\Omega}\frac{\phi_{p}}{\|\phi
_{p}\|_{\infty}}\,dx\right) ^{p-1}=1,\label{linfty2}%
\end{equation}
from what follows
\begin{equation}
\lim_{p\rightarrow1^{+}}\frac{1}{\|\phi_{p}\|_{\infty}^{p-1}}=h(\Omega
).\label{linfty3}%
\end{equation}

\end{theorem}

\begin{proof}
For each $0<k<\|\phi_{p}\|_{\infty}$, define
\[
A_{k}=\left\{ x\in\Omega\,:\,\phi_{p}>k\right\} .
\]

The function
\[
(\phi_{p}-k)^{+}=\max\left\{ \phi_{p}-k,0\right\} =\left\{
\begin{array}
[c]{ll}%
\phi_{p}-k, & \text{if \ }\phi_{p}>k\\
0, & \text{if \ }\phi_{p}\leq k
\end{array}
\right.
\]
belongs to $W_{0}^{1,p}(\Omega)$ since $\phi_{p}\in W_{0}^{1,p}(\Omega)$ and
$\phi_{p}>0$ in $\Omega$. Therefore, taking $v=(\phi_{p}-k)^{+}$ in
(\ref{aux1}), we obtain
\begin{equation}
\int_{A_{k}}|\nabla\phi_{p}|^{p}\,dx=\int_{A_{k}}(\phi_{p}-k)\,dx.\label{x1}%
\end{equation}
(Note that $A_{k}$ is an open set and therefore $\nabla(\phi_{p}-k)^{+}%
=\nabla\phi_{p}$ in $A_{k}$.)

Now, we estimate $\int_{A_{k}}|\nabla\phi_{p}|^{p}\,dx$ from below. For this
we apply H\"{o}lder inequality and Corollary \ref{explicitlower} to obtain
\[
\left( \int_{A_{k}}(\phi_{p}-k)\,dx\right) ^{p}\leq|A_{k}|^{p-1}\int_{A_{k}%
}(\phi_{p}-k)^{p}\,dx\leq\frac{|A_{k}|^{p-1}|A_{k}|^{\frac{p}{N}}}{C_{N,p}%
}\int_{A_{k}}|\nabla\phi_{p}|^{p}\,dx.
\]

Thus,
\[
C_{N,p}|A_{k}|^{-\frac{p}{N}+1-p}\left( \int_{A_{k}}(\phi_{p}-k)\,dx\right)
^{p}\leq\int_{A_{k}}|\nabla\phi_{p}|^{p}\,dx,
\]
what yields
\[
C_{N,p}|A_{k}|^{-\frac{p}{N}+1-p}\left( \int_{A_{k}}(\phi_{p}-k)\,dx\right)
^{p}\leq\int_{A_{k}}(\phi_{p}-k)\,dx.
\]

Hence, we obtain
\[
\left( \int_{A_{k}}(\phi_{p}-k)\,dx\right) ^{p-1}\leq\frac{1}{C_{N,p}}%
|A_{k}|^{\frac{p+N(p-1)}{N}},
\]
an inequality that can be rewritten as
\begin{equation}
\left( \int_{A_{k}}(\phi_{p}-k)\,dx\right) ^{\frac{N(p-1)}{p+N(p-1)}}\leq
C_{N,p}^{-\frac{N}{p+N(p-1)}}|A_{k}|.\label{x2}%
\end{equation}

Let us define
\[
f(k):=\int_{A_{k}}(\phi_{p}-k)\,dx=\int_{k}^{\infty}|A_{t}|\,dt
\]
where the last equality follows from Cavalieri's principle.

Since $f^{\prime}(k)=-|A_{k}|$, (\ref{x2}) implies that
\begin{equation}
1\leq-C_{N,p}^{-\frac{N}{p+N(p-1)}}f(k)^{-\frac{N(p-1)}{p+N(p-1)}}f^{\prime
}(k).\label{x3}%
\end{equation}

Therefore, since $f(k)>0$ and
\[
f(0)=\int_{\Omega}\phi_{p}\,dx
\]
integration of (\ref{x3}) yields an upper bound of $k$ whenever $|A_{k}|>0$:
\begin{align*}
k & \leq\frac{p+N(p-1)}{p}C_{N,p}^{-\frac{N}{p+N(p-1)}}\left[ f(0)^{\frac
{p}{p+N(p-1)}}-f(k)^{\frac{p}{p+N(p-1)}}\right] \\
& \leq\frac{p+N(p-1)}{p}C_{N,p}^{-\frac{N}{p+N(p-1)}}\left( \int_{\Omega}%
\phi_{p}\,dx\right) ^{\frac{p}{p+N(p-1)}}.
\end{align*}

This means that
\[
\Vert\phi_{p}\Vert_{\infty}\leq\frac{p+N(p-1)}{p}C_{N,p}^{-\frac{N}{p+N(p-1)}%
}\left(  \int_{\Omega}\phi_{p}\,dx\right)  ^{\frac{p}{p+N(p-1)}},
\]
which is equivalent to
\begin{equation}
\int_{\Omega}\frac{\phi_{p}}{\Vert\phi_{p}\Vert_{\infty}}\,dx\geq
C_{N,p}^{\frac{N}{p}}\left(  \frac{p}{p+N(p-1)}\right)  ^{\frac{p+N(p-1)}{p}%
}\Vert\phi_{p}\Vert_{\infty}^{\frac{N(p-1)}{p}}.\label{aux2}%
\end{equation}

Now, since (\ref{CNp}) gives that
\begin{equation}
\lim_{p\rightarrow1^{+}}C_{N,p}^{\frac{N}{p}}\left( \frac{p}{p+N(p-1)}\right)
^{\frac{p+N(p-1)}{p}}=\omega_{N}N^{N},\label{CN1}%
\end{equation}
we obtain (\ref{linfty1}), since it follows from (\ref{npbound}) and
(\ref{limhright}) that
\[
\liminf_{p\rightarrow1^{+}}\|\phi_{p}\|_{\infty}^{\frac{N(p-1)}{p}}\geq
\lim_{p\rightarrow1^{+}}\left( \frac{\|\phi_{p}\|_{1}}{|\Omega|}\right)
^{\frac{N(p-1)}{p}}=h(\Omega)^{-N}
\]

Making $p\rightarrow1^{+}$ in (\ref{linfty1}), we obtain (\ref{linfty2}),
since
\[
1=\lim_{p\rightarrow1^{+}}\left[ \omega_{N}\left( \frac{N}{h(\Omega)}\right)
^{N}\right] ^{p-1} \leq\liminf_{p\rightarrow1^{+}}\left( \int_{\Omega}%
\frac{\phi_{p}}{\|\phi_{p}\|_{\infty}}\,dx\right) ^{p-1} \leq\lim
_{p\rightarrow1^{+}}|\Omega|^{p-1}=1.
\]

At last, we obtain from (\ref{linfty2}) and (\ref{limhright}) that
\begin{align*}
\lim_{p\rightarrow1^{+}}\frac{1}{\|\phi_{p}\|_{\infty}^{p-1}} & =\lim
_{p\rightarrow1^{+}}\left( \int_{\Omega}\frac{\phi_{p}}{\|\phi_{p}\|_{\infty}%
}\,dx\right) ^{p-1}\lim\limits_{p\rightarrow1^{+}}\left( \int_{\Omega}\phi
_{p}\,dx\right) ^{1-p}\\
& =\frac{1}{\|\phi_{p}\|_{1}^{p-1}}=h(\Omega)
\end{align*}
and we are done.
\end{proof}

\begin{example}
\textrm{We take advantage of the expression (\ref{Rtorsion}) to verify
directly from (\ref{linfty3}) that $h(\Omega)=\frac{|\partial\Omega|}%
{|\Omega|}$ if $\Omega=B_{R}$, a ball of radius $R$. In fact, for this case it
follows from (\ref{linfty3}) and (\ref{Rtorsion}) that
\begin{align*}
h\left( B_{R}\right)  & =\lim_{p\rightarrow1^{+}}\|\phi_{p}\|_{\infty}^{1-p}\\
& =\lim_{p\rightarrow1^{+}}\left( \frac{p-1}{p}N^{-\frac{1}{p-1}}R^{\frac
{p}{p-1}}\right) ^{1-p}\\
& =N\lim_{p\rightarrow1^{+}}\left( \frac{p-1}{p}\right) ^{1-p}R^{-p}=\frac
{N}{R}=\frac{N\omega_{N}R^{N-1}}{\omega_{N}R^{N}}=\frac{|\partial B_{R}%
|}{|B_{R}|}.
\end{align*}
}
\end{example}

\section{Convex domains}

\label{CD} The main purpose of this section is to present a simpler proof of
(\ref{linfty2}) for the case where $\Omega$ is convex as well as to prove the
estimates
\[
\frac{1}{\|\phi_{p}\|_{\infty}^{p-1}}\leq\lambda_{p}(\Omega)\leq\frac
{1}{\|\phi_{p}\|_{\infty}^{p-1}I(q,N)^{p-1}}
\]
and
\[
\left( \frac{|\Omega|I(q,N)}{\|\phi_{p}\|_{1}}\right) ^{p-1}\leq\lambda
_{p}(\Omega) \leq\left( \frac{|\Omega|}{\|\phi_{p}\|_{1}}\right) ^{p-1}
\]
where
\[
I(q,N)=N\int_{0}^{1}(1-t)^{q}t^{N-1}\,dt\text{ \ and \ }q=\frac{p}{p-1}.
\]

For this, let $B_{R}$ be the ball centered at the origin with radius $R$ and
\[
I(\alpha,N):=N\int_{0}^{1}(1-t)^{\alpha}t^{N-1}\,dt
\]
for each $\alpha>0$ and each positive integer $N$. We remark that
\[
I(\alpha,N)=NB(\alpha-1,N)=N\frac{\Gamma(\alpha-1)\Gamma\left( N\right)
}{\Gamma(\alpha-1+N)}
\]
where $B$ and $\Gamma$ are the Beta and Gamma functions, respectively.

\begin{lemma}
For each positive integer $N$ and $\alpha>0$ one has
\begin{equation}
I(\alpha,N+1)=\frac{N+1}{N+\alpha+1}I(\alpha,N).\label{IqN}%
\end{equation}

Moreover,
\begin{equation}
\lim_{\alpha\rightarrow\infty}I(\alpha,N)^{\frac{1}{\alpha}}=1.\label{limIq}%
\end{equation}

\end{lemma}

\begin{proof}
We have
\begin{align*}
(\alpha+1)\int_{0}^{1}(1-t)^{\alpha}t^{N}\,dt & =\left[ -(1-t)^{\alpha+1}%
t^{N}\right] _{0}^{1}+\int_{0}^{1}(1-t)^{\alpha+1}Nt^{N-1}\,dt\\
& =N\int_{0}^{1}(1-t)^{\alpha+1}t^{N-1}\,dt\\
& =N\int_{0}^{1}(1-t)^{\alpha}t^{N-1}\,dt-N\int_{0}^{1}(1-t)^{\alpha}%
t^{N}\,dt\\
& =I(\alpha,N)-N\int_{0}^{1}(1-t)^{\alpha}t^{N}\,dt,
\end{align*}
thus proving (\ref{IqN}), since
\[
\frac{I(\alpha,N+1)}{N+1}=\int_{0}^{1}(1-t)^{\alpha}t^{N}\,dt=\frac
{1}{N+\alpha+1}I(\alpha,N).
\]

The proof of (\ref{limIq}) follows by induction. In fact,
\[
\lim_{\alpha\rightarrow\infty}I(\alpha,1)^{\frac{1}{\alpha}}=\lim
_{\alpha\rightarrow\infty}\left( \int_{0}^{1}(1-t)^{\alpha}\,dr\right)
^{\frac{1}{\alpha}}=\lim_{\alpha\rightarrow\infty}\left( \frac{1}{\alpha
+1}\right) ^{\frac{1}{\alpha}}=1
\]
and by assuming that $\lim\limits_{\alpha\rightarrow\infty}I(\alpha
,N)^{\frac{1}{\alpha}}=1$ we obtain from (\ref{IqN}) that
\[
\lim\limits_{\alpha\rightarrow\infty}I(\alpha,N+1)^{\frac{1}{\alpha}}%
=\lim\limits_{\alpha\rightarrow\infty}\left( \frac{N+1}{N+\alpha+1}\right)
^{\frac{1}{\alpha}}\lim\limits_{\alpha\rightarrow\infty}I(\alpha,N)^{\frac
{1}{\alpha}}=1.
\]

\vspace{-.6cm}
\end{proof}

\begin{lemma}
\label{laux2} If $\alpha>0$, then
\[
\frac{1}{|B_{R}|}\int_{B_{R}}\left( 1-\frac{|x|}{R}\right) ^{\alpha
}\,dx=I(\alpha,N).
\]

\end{lemma}

\begin{proof}
Let $\omega_{N}=|B_{1}|$. We have
\begin{align*}
\frac{1}{|B_{R}|}\int_{B_{R}}\left( 1-\frac{|x|}{R}\right) ^{\alpha}\,dx  &
=\frac{1}{R^{N}\omega_{N}}\int_{B_{1}}\left( 1-|y|\right) ^{\alpha}R^{N}\,dy\\
& =\frac{1}{\omega_{N}}\int_{0}^{1}\int_{\partial B_{r}}(1-|y|)^{\alpha}%
dS_{y}\,dr\\
& =\frac{1}{\omega_{N}}\int_{0}^{1}\left( 1-r\right) ^{\alpha}\int_{\partial
B_{r}}dS_{x}\,dr\\
& =N\int_{0}^{1}\left( 1-r\right) ^{\alpha}r^{N-1}\,dr=I(\alpha,N).
\end{align*}

\vspace{-.6cm}
\end{proof}

\begin{theorem}
Suppose that $\Omega$ is convex. Then,
\begin{equation}
I(q,N)\leq\frac{1}{|\Omega|}\int_{\Omega}\frac{\phi_{p}}{\|\phi_{p}\|_{\infty
}}\,dx\leq1,\label{main}%
\end{equation}
where $q=\dfrac{p}{p-1}$, producing a simpler proof of $(\ref{linfty2})$.
Moreover, we have
\begin{equation}
\frac{1}{\|\phi_{p}\|_{\infty}^{p-1}}\leq\lambda_{p}(\Omega)\leq\frac
{1}{\|\phi_{p}\|_{\infty}^{p-1}I(q,N)^{p-1}},\label{est1}%
\end{equation}
\begin{equation}
\left( \frac{|\Omega| I(q,N)}{\|\phi_{p}\|_{1}}\right) ^{p-1}\leq\lambda
_{p}(\Omega)\leq\left( \frac{|\Omega| }{\|\phi_{p}\|_{1}}\right)
^{p-1}\label{est2}%
\end{equation}
and also
\begin{equation}
\lim_{p\rightarrow1^{+}}\frac{1}{\|\phi_{p}\|_{\infty}^{p-1}}=h(\Omega
)=\lim_{p\rightarrow1^{+}}\frac{1}{\|\phi_{p}\|_{\infty}^{p-1}}%
.\label{cheegerlimit}%
\end{equation}

\end{theorem}

\begin{proof}
The second inequality in (\ref{main}) is obvious since $\phi_{p}\leq\|\phi
_{p}\|_{\infty}$ in $\Omega$.

For each $p>1$, take $x_{p}$ such that $\phi_{p}(x_{p})=\|\phi_{p}\|_{\infty}$
and consider the function $\Psi_{p}\in C\left( \overline{\Omega}\right) $
whose graph in $\mathbb{R}^{N}\times\mathbb{R}$ is the cone of basis $\Omega$
and height $1$\ reached at $x_{p}$ (thus, $\Psi_{p}=0$ on $\partial\Omega$ and
$\|\Psi_{p}\|_{\infty}=\Psi_{p}(x_{p})=1$).

Since $\Omega$ is convex, it follows from \cite[Thm 2]{Saka} that $\phi
_{p}^{\frac{1}{q}}$ is concave. So, we have
\[
\left( \frac{\phi_{p}}{\|\phi_{p}\|_{\infty}}\right) ^{\frac{1}{q}}\geq
\Psi_{p}\text{ in }\Omega.
\]

Now, let $R>0$ be such that $|B_{R}|=|\Omega|$ and let $\phi_{p}^{*}$ and
$\Psi_{p}^{*}$ denote the Schwarz symmetrizations of $\phi_{p}$ and $\Psi_{p}%
$, respectively. Thus, both $\phi_{p}^{*}$ and $\Psi_{p}^{*}$ are positive and
radially symmetric decreasing in $B_{R}$.

Moreover,

\begin{itemize}
\item[-] $\phi_{p}^{*}=0=\Psi_{p}^{*}$ on $\partial B_{R}$;

\item[-] $\|\phi_{p}\|_{\infty}=\|\phi_{p}^{*}\|_{\infty}=\phi_{p}^{*}(0)$;

\item[-] $\|\Psi_{p}\|_{\infty}=\|\Psi_{p}^{*}\|_{\infty}=1$;

\item[-] $\left( \phi_{p}^{\frac{1}{q}}\right) ^{*}=\left( \phi_{p}^{*}\right)
^{\frac{1}{q}}$, $\left( \Psi_{p}^{q}\right) ^{*}=\left( \Psi_{p}^{*}\right)
^{q}$;

\item[-] $\int_{\Omega}\phi_{p}\,dx=\int_{B_{R}}\phi_{p}^{*}\,dx$ and
$\int_{\Omega}\Psi_{p}\,dx=\int_{B_{R}}\Psi_{p}^{*}\,dx$.
\end{itemize}

From the definition of the Schwarz symmetrization follows that
\[
\Psi_{p}^{*}(x)=\left( 1-\frac{|x|}{R}\right) ,\text{ \ }|x|\leq R.
\]

Since Schwarz symmetrization preserves order and positive powers, we also have
that
\[
\frac{\phi_{p}^{*}(x)}{\|\phi_{p}^{*}\|_{\infty}}\geq\left( \Psi_{p}%
^{*}(x)\right) ^{q}=\left( 1-\frac{|x|}{R}\right) ^{q},\text{ \ }|x|\leq R.
\]

Thus, (\ref{main}) is consequence of Lemma \ref{laux2}, since
\begin{align*}
\frac{1}{|\Omega|}\int_{\Omega}\frac{\phi_{p}}{\|\phi_{p}\|_{\infty}}\,dx  &
=\frac{1}{|B_{R}|}\int_{B_{R}}\frac{\phi_{p}^{*}}{\|\phi_{p}^{*}\|_{\infty}%
}\,dx\\
& \geq\frac{1}{|B_{R}|}\int_{B_{R}}\left( 1-\frac{|x|}{R}\right)
^{q}\,dx=I(q,N).
\end{align*}

From (\ref{main}) and (\ref{limIq}) we obtain
\begin{align*}
1 & \geq\lim_{p\rightarrow1^{+}}\left( \int_{\Omega}\frac{\phi_{p}}{\|\phi
_{p}\|_{\infty}}\,dx\right) ^{p-1}\\
& \geq\lim_{p\rightarrow1^{+}}|\Omega|^{p-1}\lim_{p\rightarrow1^{+}%
}I(q,N)^{p-1}\\
& =\lim_{p\rightarrow1^{+}}I(q,N)^{\frac{p}{q}}=\lim_{p\rightarrow1^{+}}\left(
\lim_{q\rightarrow\infty}I(q,N)^{\frac{1}{q}}\right) ^{p}=1
\end{align*}
thus proving (\ref{linfty2}). From the last estimate and (\ref{npbound}) we
obtain (\ref{est1}), (\ref{est2}) and (\ref{cheegerlimit}).
\end{proof}

\section{Cheeger sets}

\label{CS} In this section we reproduce the current variational approach in
the $BV$ space for the Cheeger problem (\ref{cheginf}) and apply it to verify
that the $L^{1}$-normalized family $\left\{ \dfrac{\phi_{p}}{\|\phi_{p}\|_{1}%
}\right\} _{p\geq1}$ converges (up to subsequences) in $L^{1}(\Omega)$, when
$p\rightarrow1^{+}$, to a function $u\in L^{1}(\Omega)\cap L^{\infty}(\Omega)$
whose $t$-level sets $E_{t}$ are Cheeger sets. Moreover, under convexity of
$\Omega$ we verify that $u=|E_{0}|^{-1}\chi_{E_{0}}$ where $\chi_{E_{0}}$
denotes the characteristic function of the Cheeger set $E_{0}$. The function
$u$ also solves the problem
\begin{equation}
\left\{
\begin{array}
[c]{rcll}%
-\Delta_{1} & = & h(\Omega), & \text{in }\Omega\\
u & = & 0 & \text{on }\partial\Omega
\end{array}
\right. \label{1lapp}%
\end{equation}
in a sense to be clarified in the sequence (Remark \ref{solution}).

For each $v\in L^{1}(\Omega)$, let $\int_{\Omega}|Dv|\,dx$ denote the
\emph{variation} of $v$ in $\Omega$ which is defined by
\[
\int_{\Omega}|Dv|\,dx=\sup\left\{ \int_{\Omega}v\operatorname{div}g\,:\,g\in
C_{0}^{1}\left( \Omega,\mathbb{R}^{N}\right) \text{ and }\|g\|_{\infty}%
\leq1\right\} .
\]
Note that $\int_{\Omega}|Dv|\,dx$ is defined in terms of the weak
(distributional) derivative of $u$. Moreover, the variation of a function
$v\in C^{1}(\Omega)$ coincides with the $L^{1}$-norm of its gradient, that is
\[
\int_{\Omega}|Dv|\,dx=\int_{\Omega}|\nabla v|\,dx\text{ \ when }v\in
C^{1}(\Omega).
\]

The space $BV(\Omega)$ of the \emph{bounded variation functions} is then
defined by
\[
BV(\Omega)=\left\{ v\in L^{1}(\Omega)\,:\,\int_{\Omega}|Dv|\,dx<\infty\right\}
.
\]

It is known (see \cite{Evans}, \cite{Giusti}) that $BV(\Omega)$ is a Banach
space with the norm
\[
\|v\|_{BV}:=\int_{\Omega}|v|\,dx+\int_{\Omega}|Dv|\,dx
\]
and, moreover, the following properties hold (see \cite[Section 5.2]{Evans}):

\begin{lemma}
[lower semicontinuity]\label{lsc} If $v_{n}\rightarrow v$ in $L^{1}(\Omega)$
then
\[
\int_{\Omega}|Dv|\,dx\leq\liminf\limits_{n}\int_{\Omega}|Dv_{n}|\,dx.
\]

\end{lemma}

\begin{lemma}
[$L^{1}$-compactness]\label{L1compac} If $\left\{ v_{n}\right\} _{n\in
\mathbb{N}}\subset BV(\Omega)$ is a bounded sequence in the $BV$-norm, then
(up to a subsequence) $v_{n}\rightarrow v$ in $L^{1}(\Omega)$.
\end{lemma}

\begin{lemma}
[coarea formula]\label{coarea} Let $v\in BV(\Omega)$. Then
\[
\int_{\Omega}|Dv|\,dx=\int_{-\infty}^{\infty}|\partial E_{t}|\,dt.
\]
$($Here $E_{t}:=\left\{ x\in\Omega\,:\,v(x)>t\right\} $ is the $t$-level set
of $v$ and $|\partial E_{t}|$ denotes its perimeter in $\Omega.)$
\end{lemma}

It is also known that when $\partial\Omega$ is Lipschitz, functions in
$BV(\Omega)$ have a trace on $\partial\Omega$. Thus, from now on we assume
that $\partial\Omega$ is Lipschitz.

Since a Cheeger set $E\subset\overline{Omega}$ touches $\partial\Omega$ it is
important to consider the boundary $\partial\Omega$ in the variational
formulation of the Cheeger problem.

We consider the minimizing problem
\begin{equation}
\mu=\inf_{v\in\Lambda}H(v)\label{infH}%
\end{equation}
where
\begin{equation}
H(v):=\int_{\Omega}|Dv|\,dx+\int_{\partial\Omega}|v|\,d\mathcal{H}%
^{N-1}\label{H(v)}%
\end{equation}
and
\[
\Lambda=\left\{ v\in BV\left( \mathbb{R}^{N}\right) \,:\, v\geq0\ \text{in}%
\ \overline{\Omega},\text{ \ }v\equiv0\ \text{in}\ \mathbb{R}^{N}%
\setminus\overline{\Omega},\ \|v\|_{1}=1\right\} .
\]

In the surface integral in (\ref{H(v)}), $|v|$ denotes the internal trace of
$v$ and $d\mathcal{H}^{N-1}$ denotes the $(N-1)$-dimensional Hausdorff measure.

We also remark (see \cite{Evans}) that $H\left( \chi_{E}\right) $ is the
perimeter of $E$ in $\mathbb{R}^{N}$ for $E\subset\overline{\Omega}$ and that
if $v\in\Lambda$ then $v\in BV\left( \mathbb{R}^{N}\right) $ and
\[
\int_{\mathbb{R}^{N}}|Dv|\,dx=\int_{\Omega}|Dv|\,dx+\int_{\partial\Omega
}|v|\,d\mathcal{H}^{N-1}.
\]

\begin{proposition}
It holds $\mu=h(\Omega)$.
\end{proposition}

\begin{proof}
For an arbitrary $E\subset\overline{\Omega}$ we have
\[
\frac{|\partial E|}{|E|}=\frac{H\left( \chi_{E}\right) }{|E|}=H\left(
\frac{\chi_{E}}{|E|}\right) \geq\mu
\]
what implies, in view of (\ref{cheginf}), that $\mu\leq h(\Omega)$. On the
other hand, if $v\in\Lambda$ \ it follows from Lemma \ref{coarea} and
Cavalieri's principle that
\begin{align*}
H(v) & =\int_{\mathbb{R}^{N}}|Dv|\,dx\\
& =\int_{0}^{\infty}|\partial E_{t}|\,dt\\
& =\int_{0}^{\infty}\frac{|\partial E_{t}|}{|E_{t}|}|E_{t}|\,dt\\
& \geq h(\Omega)\int_{0}^{\infty}|E_{t}|\,dt=h(\Omega)\|v\|_{1}=h(\Omega).
\end{align*}
Since $v\in\Lambda$ is arbitrary, we conclude from (\ref{infH}) that
$h(\Omega)\leq\mu$.
\end{proof}

\begin{remark}
\label{solution}\textrm{Since $\mu=h(\Omega)$, the problem (\ref{infH}) can be
considered as a variational formulation of (\ref{1lapp}). In view (\ref{kaw})
such a solution is considered as an eigenvalue of (\ref{1lapp}). For details
we refer to \cite[Remark 7]{Kaw2}. }
\end{remark}

The existence of a Cheeger set $E\subset\overline{\Omega}$ is equivalent to
finding a minimizer $u$ for the problem (\ref{infH}) in the following sense:

\begin{proposition}
\label{Cheegsets} If $u$ minimizes $(\ref{infH})$, then its $t$-level sets
\[
E_{t}:=\left\{ x\in\Omega\,:\,u(x)>t\right\}
\]
satisfying $|E_{t}|>0$ are Cheeger sets. In particular, $E_{0}$ is a Cheeger set.

On the other hand, if $E\subset\overline{\Omega}$ is a Cheeger set, then
$\frac{\chi_{E}}{|E|}$ minimizes $(\ref{infH})$.
\end{proposition}

\begin{proof}
[Proof (sketch)]For the first claim we present only a sketch and refer to
\cite[Theor. 2]{Carlier} for details.

Let $u\in\Lambda$ be a minimizer of (\ref{infH}) and define
\[
T_{n}(v)=\left\{
\begin{array}
[c]{ll}%
0 & \text{if }0<v\\
nv & \text{if }0\leq v<\frac{1}{n}\\
1 & \text{if }v\geq\frac{1}{n}.
\end{array}
\right.
\]
For $n$ large enough the function $w_{n}=\frac{T_{n}(u)}{\left\Vert
T_{n}(u)\right\Vert _{1}}$ also minimizes (\ref{infH}) in $\Lambda$. Hence the
convergence in $L^{1}$ of $w_{n}$ to $w_{0}:=\frac{\chi_{E_{0}}}{|E_{0}|}%
\in\Lambda$ yields that $w_{0}$ solves (\ref{infH}). Therefore,
\[
h(\Omega)=H(w_{0})=\frac{1}{|E_{0}|}H(\chi_{E_{0}})=\frac{|\partial E_{0}%
|}{|E_{0}|}%
\]
proving that $E_{0}$ is a Cheeger set.

If $t>0$ is such that $|E_{t}|>0$ then it is possible to verify that the
function $v:=\frac{(u-t)_{+}}{\left\Vert (u-t)_{+}\right\Vert _{1}}$ solves
(\ref{infH}). Thus, by applying the previous argument for $E_{0}^{v}$, the
zero-level set of $v$, we conclude that $\frac{\chi_{E_{0}^{v}}}{|E_{0}^{v}|}$
also solves (\ref{infH}). Since
\[
\frac{\chi_{E_{0}^{v}}}{|E_{0}^{v}|}=\frac{\chi_{E_{t}}}{|E_{t}|}%
\]
we are done.

Now, in order to prove the second claim, let $E$ be a Cheeger set and take
$u=\frac{\chi_{E}}{|E|}$. Then, $u\in\Lambda$ and
\[
h(\Omega)=\frac{|\partial E|}{|E|}=\frac{H\left( \chi_{E}\right) }%
{|E|}=H\left( \frac{\chi_{E}}{|E|}\right) =H(u).
\]

\vspace{-.6cm}
\end{proof}

Now we prove our main result on Cheeger sets and the minimization of $H$.

\begin{theorem}
Let $u_{p}:=\dfrac{\phi_{p}}{\|\phi_{p}\|_{1}}$. Then there exists a sequence
$\left\{ u_{p_{n}}\right\} \subset C_{0}^{1}\left( \overline{\Omega}\right)
\cap BV(\Omega)$ and a function $u\in\Lambda\cap L^{\infty}(\Omega)$, such
that $p_{n}\rightarrow1^{+}$ and $u_{p_{n}}\rightarrow u$ in $L^{1}(\Omega)$. Moreover:

\begin{enumerate}
\item[$(i)$] $u=0$ on $\partial\Omega$;

\item[$(ii)$] $0\leq u\leq\omega_{N}^{-1}\left( \dfrac{h(\Omega)}{N}\right)
^{N}$ in $\Omega$;

\item[$(iii)$] $h(\Omega)=H(u)$, that is, $u$ minimizes $(\ref{infH})$;

\item[$(iv)$] Almost all $t$-level sets of $u$ are Cheeger sets for $0\leq
t\leq\|u\|_{\infty}$.
\end{enumerate}
\end{theorem}

\begin{proof}
Since
\[
\int_{\Omega}|\nabla\phi_{p}|^{p}\,dx=\int_{\Omega}\phi_{p}\,dx
\]
we have that
\[
\int_{\Omega}|\nabla u_{p}|^{p}\,dx=\frac{1}{\Vert\phi_{p}\Vert_{1}^{p-1}}%
\int_{\Omega}u_{p}\,dx=\frac{1}{\Vert\phi_{p}\Vert_{1}^{p-1}}.
\]
Thus, it follows from H\"{o}lder inequality that
\[
\int_{\Omega}|\nabla u_{p}|\,dx\leq\left(  \int_{\Omega}|\nabla u_{p}%
|^{p}\,dx\right)  ^{\frac{1}{p}}|\Omega|^{1-\frac{1}{p}}=\left(  \frac
{1}{\Vert\phi_{p}\Vert_{1}^{p-1}}\right)  ^{\frac{1}{p}}|\Omega|^{1-\frac
{1}{p}}.
\]

Hence, since $u_{p}\in C^{1,\beta}\left( \overline{\Omega}\right) \cap
W_{0}^{1,p}(\Omega)\subset C_{0}^{1}\left( \overline{\Omega}\right) $ (here
$\beta$ may depend on $p$) and $\|u_{p}\|_{1}=1$, we have
\begin{align*}
\|u_{p}\|_{BV} & =\int_{\Omega}u_{p}\,dx+\int_{\Omega}|Du_{p}|\,dx\\
& =1+\int_{\Omega}|\nabla u_{p}|\,dx\\
& \leq1+\left( \frac{1}{\|\phi_{p}\|_{1}^{p-1}}\right) ^{\frac{1}{p}}%
|\Omega|^{1-\frac{1}{p}}
\begin{array}
[c]{c}%
\\
\overrightarrow{p\rightarrow1^{+}}%
\end{array}
1+h(\Omega)<\infty.
\end{align*}

Therefore the family $\left\{  u_{p}\right\}  _{p\geq1}$ is a bounded in
$BV(\Omega)$ for all $p$ sufficiently close to $1^{+}$. Thus, it follows from
Lemma \ref{L1compac} that there exists a sequence $p_{n}\rightarrow1^{+}$ such
that
\[
u_{p_{n}}\rightarrow u\text{ \ in \ }L^{1}(\Omega).
\]
Moreover, $\Vert u\Vert_{1}=1$ and, up to a subsequence, we can assume that
$u_{p_{n}}\rightarrow u$ a.e. in $\Omega$ and that $u$ satisfies properties
$(i)$ and $(ii)$, the upper bound in $(ii)$ being a consequence of
(\ref{linfty1}).

Lemma \ref{lsc} applied to the sequence $\left\{ u_{p_{n}}\right\} $ yields
\begin{align*}
\int_{\Omega}|Du|\,dx & \leq\liminf\limits_{n}\int_{\Omega}|Du_{p_{n}}|\,dx\\
& =\liminf\limits_{n}\int_{\Omega}|\nabla u_{p_{n}}|\,dx\\
& \leq\lim_{n}\left( \int_{\Omega}|\nabla u_{p_{n}}|^{p_{n}}\,dx\right)
^{\frac{1}{p_{n}}}|\Omega|^{1-\frac{1}{p_{n}}}\\
& =\lim_{n}\left( \frac{1}{\|\phi_{p}\|_{1}^{p_{n}}}\int_{\Omega}|\nabla
\phi_{p_{n}}|^{p}\,dx\right) ^{\frac{1}{p_{n}}}\\
& =\lim_{n}\left( \frac{1}{\|\phi_{p}\|_{1}^{p_{n}}}\int_{\Omega}\phi_{p_{n}%
}\,dx\right) ^{\frac{1}{p_{n}}}=\lim_{n} \left( \frac{1}{\|\phi_{p}%
\|_{1}^{p_{n}-1}}\right) ^{\frac{1}{p_{n}}}=h(\Omega).
\end{align*}
Thus, $u\in\Lambda$ and, since $u=0$ on $\partial\Omega$, we have
\begin{align*}
H(u) & =\int_{\Omega}|Du|\,dx+\int_{\partial\Omega}|u|\,d\mathcal{H}^{N-1}\\
& =\int_{\Omega}|Du|\,dx\leq h(\Omega)=\inf\limits_{v\in\Lambda}H(v)\leq H(u).
\end{align*}
Hence, $H(u)=h(\Omega)$, that is, $u$ is a minimizer of (\ref{infH}), proving
$(iii)$.

The claim $(iv)$ is consequence of $(iii)$ and Proposition \ref{Cheegsets}.
\end{proof}

\begin{remark}
\textrm{If $\Omega$ is convex, then the function $u$ of the last theorem can
be written as
\[
u=\|u\|_{\infty}\chi_{E_{0}}=\frac{\chi_{E_{0}}}{|E_{0}|}
\]
where $E_{0}=\left\{ x\in\overline{\Omega}\,:\,u(x)>0\right\} $. In fact, this
follows from the uniqueness of the Cheeger set, since $E_{0}=E_{t}$ for almost
all $t$-level set $E_{t}$ of $u$, with $0\leq t\leq\|u\|_{\infty}$. Thus,
since $\|u\|_{\infty}\chi_{E_{0}}\geq u$ in $E_{0}$ we have
\begin{align*}
\|\|u\|_{\infty}\chi_{E_{0}}-u\|_{1} & =\int_{\Omega}|\|u\|_{\infty}%
\chi_{E_{0}}-u|\,dx\\
& =\int_{E_{0}}\left( \|u\|_{\infty}\chi_{E_{0}}-u\right) \,dx\\
& =\|u\|_{\infty}|E_{0}|-\int_{0}^{\|u\|_{\infty}}|E_{t}|\,dt\\
& =\|u\|_{\infty}|E_{0}|-\int_{0}^{\|u\|_{\infty}}|E_{0}|\,dt=0.
\end{align*}
}

\textrm{Since $\|u\|_{1}=1$ we also have $1=\|u\|_{\infty}\|\chi_{E_{0}}%
\|_{1}=\|u\|_{\infty}|E_{0}|$ implying that $\|u\|_{\infty}=\dfrac{1}{|E_{0}%
|}$. }
\end{remark}

Since $\frac{\chi_{E_{0}}}{|E_{0}|}$ is a Cheeger set for $E_{0}=\left\{
x\in\overline{\Omega}\,:\,u(x)>0\right\} $, it is interesting to notice that
the claim $(ii)$ gives a lower bound for the volume $|E_{0}|$ in terms of the
Cheeger constant. In fact,
\[
1=\|u\|_{1}=\int_{E_{0}}u\,dx\leq|E_{0}|\|u\|_{\infty}\leq|E_{0}|\omega
_{N}^{-1}\left( \dfrac{h(\Omega)}{N}\right) ^{N}
\]
implies that
\[
\omega_{N}\left( \frac{N}{h(\Omega)}\right) ^{N}\leq|E_{0}|
\]
or, what is the same,
\begin{equation}
|B_{1}|h(B_{1})^{N}\leq|E_{0}|h(\Omega)^{N}\label{lowerE}%
\end{equation}
since $\omega_{N}=|B_{1}|$ and $h(B_{1})=N$.

Moreover, since $h(\Omega)|E_{0}|=|\partial E_{0}|$, we also have
\[
h(\Omega)|B_{1}|\left( \frac{N}{h(\Omega)}\right) ^{N}\leq|\partial E_{0}|.
\]

\begin{example}
\textrm{As pointed out in \cite{Kaw2}, if $\Omega=[-1,1]\times[-1,1]$ is the
square, then
\[
h(\Omega)=\frac{4-\pi}{4-2\sqrt{\pi}},
\]
and the (unique) Cheeger set $E$ satisfies
\[
|E|=4-\frac{\left( 4-2\sqrt{\pi}\right) ^{2}}{4-\pi}\approx3.7587
\]
and
\[
|\partial E|=8-\frac{(8-2\pi)(4-2\sqrt{\pi})}{4-\pi}\approx7.0898
\]
Thus, we can evaluate (\ref{lowerE}):
\[
\omega_{2}\left( \frac{2}{h(\Omega)}\right) ^{2}=4\pi\left( \frac{4-2\sqrt
{\pi}}{4-\pi}\right) ^{2}\approx3.532<3.7587\approx|E|
\]
and
\[
h(\Omega)\omega_{2}\left( \frac{2}{h(\Omega)}\right) ^{2}=\frac{4\pi}%
{h(\Omega)}=\frac{4\pi(4-2\sqrt{\pi})}{4-\pi}\approx6.6622<7.0898\approx
|\partial E|.
\]
}
\end{example}

\begin{remark}
\textrm{We remark that (\ref{lowerE}) is optimal if $\Omega=B_{R}$ is a ball
since $E=B_{R}$ is the only Cheeger set and
\[
|B_{R}|h(B_{R})^{N}=\omega_{N}R^{N}\left( \frac{N}{R}\right) ^{N}=\omega
_{N}N=|B_{1}|h(B_{1}).
\]
}
\end{remark}

\begin{remark}
\textrm{Taking into account Theorem \ref{linfty}, the $L^{\infty}%
$-normalization of $\phi_{p}$ also produces, when $p\rightarrow1^{+}$, a
function $u\in BV(\Omega)\cap L^{\infty}(\Omega)$ such that $\|u\|_{\infty
}\leq1$,
\[
\omega_{N}\left( \frac{N}{h(\Omega)}\right) ^{N}\leq\|u\|_{1}\leq|\Omega|
\]
and whose $t$-level sets are Cheeger sets almost all $0\leq t\leq1$. }

\textrm{Moreover, $u$ satisfies
\[
h(\Omega)=\frac{H(u)}{\int_{\Omega}u\,dx}\leq\frac{H(v)}{\int_{\Omega}%
v\,dx}\text{ for all }v\in BV(\Omega)\text{ satisfying }0\leq v\leq1\text{
\ in \ }\overline{\Omega}.
\]
}
\end{remark}

\end{document}